\documentclass[11pt,reqno]{amsart}

\usepackage[T1]{fontenc}
\usepackage{mathpazo}
\usepackage{microtype}
\usepackage[a4paper,margin=31mm]{geometry}
\usepackage{amsmath,amssymb,mathtools}
\usepackage{enumitem}
\usepackage[dvipsnames]{xcolor}
\usepackage{hyperref}

\hypersetup{
  hidelinks,
  pdftitle={Small Kazhdan constants for small generating sets of linear groups},
  pdfauthor={Tsachik Gelander, Alexander Lubotzky, and Jvbin Yao},
  pdfsubject={Kazhdan constants and bounded-size generating sets},
  pdfkeywords={Kazhdan constant, property (T), linear group, Schottky group, ping-pong}
}

\numberwithin{equation}{section}

\theoremstyle{plain}
\newtheorem{theorem}{Theorem}[section]
\newtheorem{proposition}[theorem]{Proposition}
\newtheorem{lemma}[theorem]{Lemma}
\newtheorem{corollary}[theorem]{Corollary}

\theoremstyle{remark}
\newtheorem{remark}[theorem]{Remark}

\newcommand{\Zcl}[1]{\overline{#1}^{\,\mathrm{Zar}}}
\newcommand{\kk}{\kappa}
\DeclareMathOperator{\Opp}{Opp}
\DeclareMathOperator{\Int}{Int}

\title[Kazhdan constants for small generating sets]
{Small Kazhdan Constants for Small\\
Generating Sets of Linear Groups}

\author{Tsachik Gelander}
\address{{Department of Mathematics,
Northwestern University, 2033 Sheridan Road,
Evanston, IL 60208, USA}}
\email{{tsachik.gelander@northwestern.edu}}
\author{Alexander Lubotzky}
\address{{Department of Mathematics,
Weizmann Institute of Science, 234 Herzl Street,
Rehovot 7610001, Israel}}
\email{{alex.lubotzky@mail.huji.ac.il}}
\author{{Jvbin Yao}}
\address{Research Center for Operator Algebras,
School of Mathematical Sciences,
East China Normal University,
Shanghai 200241, PR China}
\email{52285500009@stu.ecnu.edu.cn}
\date{Preliminary draft -- 19 September 2026}


\begin{document}

\begin{abstract}
Answering a question of {Lubotzky and Yao}, we prove that if
$\Gamma$ is an infinite finitely generated linear group over an
arbitrary field and $d(\Gamma)$ is the minimal cardinality of a
generating set, then
\[
   \kk_{\le d(\Gamma)+2}(\Gamma)=0.
\]
Here $\kk_{\le r}(\Gamma)$ denotes the infimum of the Kazhdan
constants over generating sets of cardinality at most $r$.
{
We also prove the same conclusion for every finitely generated group
admitting a non-elementary convergence action on a compact metrizable
space.  This includes non-elementary word-hyperbolic and relatively
hyperbolic groups, as well as finitely generated non-elementary subgroups
of word-hyperbolic groups.
}
\end{abstract}

\maketitle

\section{Introduction}\label{sec:introduction}

Let $\Gamma$ be a finitely generated group and let $S\subseteq\Gamma$ be a
finite, not necessarily symmetric, generating set.  Its Kazhdan constant is
\[
 \kk(\Gamma,S)=
 \inf_{\substack{\pi:\Gamma\to\mathcal U(\mathcal H)\\
                  \mathcal H^\Gamma=\{0\}}}
 \ \inf_{\|\xi\|=1}\ \max_{s\in S}\|\pi(s)\xi-\xi\|.
\]
Thus $\Gamma$ has property~$(T)$ if and only if
$\kk(\Gamma,S)>0$ for some, equivalently every, finite generating set $S$.
For $r\ge d(\Gamma)$, where $d(\Gamma)$ is the minimal number of generators
of $\Gamma$, put
\[
 \kk_{\le r}(\Gamma)=
 \inf_{\substack{\langle S\rangle=\Gamma\\|S|\le r}}
 \kk(\Gamma,S).
\]
Thus the usual uniform Kazhdan constant is
$\kk_{\mathrm u}(\Gamma)=\inf_r\kk_{\le r}(\Gamma)$.  The equality
$\kk_{\mathrm u}(\Gamma)=0$ does not, by itself, imply that
$\kk_{\le r}(\Gamma)=0$ for any fixed $r$.

We use the following terminology for cardinality bounds in such
vanishing results.  The generating sets are \emph{bounded} if their
cardinalities are bounded by a number depending on $\Gamma$, independently
of how small the Kazhdan constant is required to be.  They are \emph{small}
if this bound is $f(d(\Gamma))$, for a single function $f$ independent of
$\Gamma$, and \emph{minimal} if their cardinality is $d(\Gamma)$.
Thus a minimal result is stronger than a small result, which in turn is
stronger than a bounded result.

Gelander and \.{Z}uk \cite{GZ} answered a question of Lubotzky by exhibiting
groups with property~$(T)$ and zero uniform Kazhdan constant.  For groups that embed densely in a
connected Lie group, their argument already yields generating sets of
uniformly bounded cardinality.  More precisely, their
construction for a dense subgroup of a connected semisimple Lie group
already gives the bound $d(\Gamma)+2$.  Two auxiliary generators are
adjoined and used to move the original $d(\Gamma)$ generators into an
arbitrarily small neighborhood of the identity in the ambient Lie group,
while preserving generation of $\Gamma$.  Our construction has the same
general form: we adjoin $t,b$ and use words in them to modify the original
generators, again obtaining the bound $d(\Gamma)+2$.

Osin and Sonkin \cite{OS}
constructed infinite finitely generated groups with
$\kk_{\mathrm u}(\Gamma)>0$.  Recently, {the second and third authors}
\cite{LY} proved that $\kk_{\mathrm u}(\Gamma)=0$ for every infinite
finitely generated linear group, over an arbitrary field.  Their argument
allows the number of generators to grow as the Kazhdan constant
decreases.  They asked whether vanishing can be achieved with
\emph{bounded} generating sets, allowing the bound to depend on $\Gamma$.
{The third author} \cite{Yao} subsequently proved the minimal result
\[
 \kk_{\le2}\bigl(\mathrm{SL}_n(\mathbb Z)\bigr)=0,
 \qquad n\ge3.
\]
Indeed, $d(\mathrm{SL}_n(\mathbb Z))=2$ for $n\ge3$, so this
achieves vanishing over generating sets of minimal cardinality.

We prove a \emph{small} generating-set result for every infinite
finitely generated linear group over an arbitrary field, with the explicit
bound $d(\Gamma)+2$.  This strengthens the bounded generating-set
assertion asked for by {Lubotzky and Yao}.

\begin{theorem}\label{thm:main}
Let $F$ be a field and let
$\Gamma\le\mathrm{GL}_n(F)$ be infinite and finitely generated.  If
$d=d(\Gamma)$, then
\[
 \kk_{\le d+2}(\Gamma)=0.
\]
\end{theorem}

If $\Gamma$ does not have property~$(T)$, then $\kk(\Gamma,S)=0$ for every
finite generating set $S$, and the theorem is immediate.  The content is
therefore the property-$(T)$ case: the constants of the individual
generating sets are positive, but their infimum is zero.

The two additional generators in our construction are the
auxiliary elements $t$ and $b$; the other $d(\Gamma)$ generators are
modifications of a minimal generating set.

The general \emph{minimal} question remains open: does
\[
 \kk_{\le d(\Gamma)}(\Gamma)=0
\]
hold for every infinite finitely generated linear group $\Gamma$ over an
arbitrary field?

{The construction of Lubotzky and Yao} in \cite{LY} uses representations
with finite image and gives vanishing of the corresponding uniform
$\tau$-constant.  Our proof uses quasi-regular representations on infinite
coset spaces; it does not establish the analogous bounded-cardinality
statement when one restricts to representations with finite image.

{
The same bound also holds for groups admitting a non-elementary
convergence action, without any linearity assumption.  In
Section~\ref{sec:convergence-extension} we prove the following theorem.

\begin{theorem}[Convergence-group extension]\label{thm:intro-convergence}
Let $\Gamma$ be a finitely generated group admitting a non-elementary
convergence action on a compact metrizable space.  If $d=d(\Gamma)$, then
\[
   \kk_{\le d+2}(\Gamma)=0.
\]
\end{theorem}

This includes non-elementary word-hyperbolic and relatively hyperbolic
groups, as well as finitely generated non-elementary subgroups of
word-hyperbolic groups; see Theorem~\ref{thm:convergence-extension} and
Corollary~\ref{cor:nonlinear-examples}.
\par}


{
\subsection{Sketch of proof}\label{subsec:proof-sketch}
We explain the algebraic strategy, assuming that $\Gamma$ has
property~$(T)$.  The key to obtaining a small Kazhdan constant is to
construct a long segment in a coset space on which all but one of the
generators act trivially, while the remaining generator shifts the
segment by one step.  More precisely, suppose that
$\Gamma=\langle t,Y\rangle$ and set
\[
 H_M=\langle t^jyt^{-j}:y\in Y,\ 0\le j\le M\rangle.
\]
Lemma~\ref{lem:window}, extending Yao's quasi-regular
estimate in \cite{Yao}, shows that if $H_M$ has infinite index, then the
cosets $H_Mt^j$, $0\le j\le M$, are distinct.  Their normalized sum in
$\ell^2(H_M\backslash\Gamma)$ is fixed by every $y\in Y$, whereas applying
$t$ changes only the two endpoint terms.  This gives
$\kk(\Gamma,\{t\}\cup Y)\le\sqrt{2/(M+1)}$.
We will arrange that $H_M$ is a nonabelian free group; it then has
infinite index, since a finite-index subgroup would inherit
property~$(T)$.

Start with a minimal generating tuple $s_1,\ldots,s_d$ and adjoin two
suitably chosen elements $t,b$ generating a Schottky free group
$B=\langle t,b\rangle$.  The element $t$ will perform the shift, while
$b$ supplies the conjugates
\[
 b_j=t^jbt^{-j},\qquad j\in\mathbb Z.
\]
These form a free basis of the kernel of the map $B\to\mathbb Z$
sending $t$ to $1$ and $b$ to $0$.  Thus
$J_M=\langle b_0,\ldots,b_M\rangle$ is already free.  For each original
generator choose $u_i=b_{r_i}$ and $v_i=b_{-q_i}$, and replace $s_i$ by
\[
 x_i(N)=u_i^Ns_iv_i^N.
\]
Because $u_i,v_i$ are words in the retained generators $t,b$, these
replacements are achieved by Nielsen moves on the enlarged tuple.
{In particular, we recover $s_i=u_i^{-N}x_i(N)v_i^{-N}$,
and so $S_{M,N}=\{t,b,x_1(N),\ldots,x_d(N)\}$ still generates
$\Gamma$ and has at most $d+2$ elements.}  Adding $t,b$ therefore gives us freedom to
modify all the original generators without losing generation or
increasing the cardinality further.

The integers $r_i,q_i$ separate the corrections and all their relevant
conjugates.  Indeed,
\[
 t^jx_i(N)t^{-j}
 =b_{r_i+j}^{\,N}(t^js_it^{-j})b_{j-q_i}^{\,N}.
\]
For fixed $M$, we choose the positive integers $r_i,q_i$ sufficiently
large that the $2d$ integer intervals
\[
 [r_i,r_i+M],\qquad[-q_i,M-q_i],\qquad 1\le i\le d,
\]
are pairwise disjoint and avoid $[0,M]$.  There are two intervals for
each original generator because we must control both the forward and
the inverse action: the left correction determines the limiting
attracting point, while the right correction determines the limiting
repelling point.  For $t^jx_i(N)t^{-j}$ these points are labelled by
$b_{r_i+j}^{+}$ and $b_{j-q_i}^{-}$ in the boundary of $B$.  As $j$ runs
from $0$ to $M$, their indices run through the two assigned intervals.
The disjoint intervals ensure that all these points are distinct.
Moreover, the free-basis decomposition shows that they lie outside
$\partial J_M$, have trivial stabilizers in $J_M$, and have pairwise
disjoint $J_M$-orbits.  This is the separation needed for ping-pong.

The required dynamical ingredient is a finite-family \emph{relative}
ping-pong statement: a family of elements with sufficiently strong
forward and inverse contraction towards such separated points can be
adjoined freely to $J_M$ (Proposition~\ref{prop:relative-pp}).  The
initial choice of $t,b$ and sufficiently large $r_i,q_i$ ensures that
the two-sided products above acquire this contraction as $N\to\infty$.
Hence, for large $N$, the group
\[
 H_{M,N}=\langle b_j,\ t^jx_i(N)t^{-j}:
                 0\le j\le M,\ 1\le i\le d\rangle
\]
is the free product of $J_M$ and the infinite cyclic groups generated
by the displayed conjugates of the $x_i(N)$.  It is therefore free and
has infinite index.  Applying Lemma~\ref{lem:window} to $S_{M,N}$ now
gives the required bound, which tends to zero as $M\to\infty$.

{
The dynamical argument for linear groups begins with a reduction to
local fields.  We first restrict to the finitely generated coefficient
field determined by the entries of the generators and their inverses.
Proposition~\ref{prop:projective-input} then supplies an embedding of
this field into a local field $k$ and a suitable projective
representation of the Zariski closure over $k$.  This allows us to use
compactness of projective space over $k$ and projective contraction.
\par}

What makes this construction possible for linear groups is the refined
projective dynamics developed by Breuillard and the first author in
\emph{On dense free subgroups of Lie groups}~\cite{BG-dense} and
\emph{A topological Tits alternative}~\cite{BG-topological}, and by the
first author and Glasner in \emph{Countable primitive groups}~\cite{GG}.
These tools supply the Schottky group and the control of the two-sided
products required above.  Non-elementary convergence actions on compact
metrizable spaces fit precisely the same framework: the algebraic
argument is unchanged, and the dynamical part is much simpler, since
opposition is replaced by distinctness of points and the needed
contraction follows directly from convergence dynamics.  We carry out
this extension in Section~\ref{sec:convergence-extension}.
\par}

\section{The projective dynamics and algebraic input}
\label{sec:algebraic}

For the remainder of the proof, assume that $\Gamma$ has property~$(T)$ and
fix a minimal generating tuple $s_1,\ldots,s_d$.  Replacing the coefficient
field by the subfield $K$ generated over its prime field by the matrix entries
of the $s_i$ and their inverses, we may assume that the field $K$ is
finitely generated.  Put
\[
 \mathbf H=\Zcl{\Gamma},
 \qquad
 \Gamma^\circ=\Gamma\cap\mathbf H^\circ.
\]

\begin{lemma}\label{lem:nonsolvable-closure}
The identity component $\mathbf H^\circ$ is not solvable, and
$\Gamma^\circ$ is Zariski dense in $\mathbf H^\circ$.
\end{lemma}

\begin{proof}
If $\mathbf H^\circ$ were solvable, then $\Gamma$ would be virtually
solvable and hence amenable.  An amenable discrete group with
property~$(T)$ is finite, contrary to our hypothesis.

The subgroup $\Gamma^\circ$ has finite index in $\Gamma$.  Taking Zariski
closures in a finite coset decomposition of $\Gamma$ shows that
$\Zcl{\Gamma^\circ}$ has the same dimension as $\mathbf H$.  Since it is a
closed subgroup of $\mathbf H^\circ$, it equals $\mathbf H^\circ$.
\end{proof}

\paragraph{{\textbf{Projective-dynamics terminology.}}}
Let $k$ be a local field and let $V$ be a finite-dimensional $k$-vector
space.  We let $\operatorname{PGL}(V)$ act on $\mathbb P(V^*)$
contragrediently.  For $[\varphi]\in\mathbb P(V^*)$, write
\[
 [\varphi]^\perp=\mathbb P(\ker\varphi)\subset\mathbb P(V).
\]
We use \emph{proximal} and \emph{very proximal} in the sense of
\cite[Section~3]{BG-topological} and \cite[Section~7.2]{GG}.
For the standard projective metric $d$, write $N_\varepsilon(A)$ for the
$\varepsilon$-neighborhood of a subset $A$.
A transformation $g$ is $(r,\varepsilon)$-proximal if $r>2\varepsilon>0$
and there are a point $v_g$ and a hyperplane $H_g$ such that
\[
 d(v_g,H_g)\ge r,
 \qquad
 g\bigl(\mathbb P(V)\setminus N_\varepsilon(H_g)\bigr)
 \subset N_\varepsilon(\{v_g\}).
\]
It is proximal if this holds for some such $r,\varepsilon$.
It is $(r,\varepsilon)$-very proximal if both $g$ and $g^{-1}$ are
$(r,\varepsilon)$-proximal, and very proximal if both are proximal.
For the proximal
transformations occurring below, the parameters are chosen so that
\cite[Lemma~3.2]{BG-topological} applies.  We denote the canonical
attracting point and canonical repelling hyperplane of $g$ by
$\overline v_g$ and $\overline H_g$, respectively.  Thus the positive
powers of $g$ converge locally uniformly to $\overline v_g$ off
$\overline H_g$, and positive powers have the same canonical attracting
point and repelling hyperplane.

The partial flag variety of incident lines and hyperplanes is
\[
 \mathcal I(V)=
 \bigl\{(\ell,[\varphi])\in\mathbb P(V)\times\mathbb P(V^*):
          \ell\subset[\varphi]^\perp\bigr\}.
\]
We identify $[\varphi]$ with the hyperplane
$H=[\varphi]^\perp$ and write such a partial flag as $(\ell,H)$.  If $g$
is very proximal, then
$\overline v_g\in\overline H_{g^{-1}}$ and
$\overline v_{g^{-1}}\in\overline H_g$.  The two partial flags associated
with $g$ are therefore
\[
 \mathfrak f^+(g)=(\overline v_g,\overline H_{g^{-1}}),
 \qquad
 \mathfrak f^-(g)=(\overline v_{g^{-1}},\overline H_g).
\]
For $m\ge1$ and $h\in\operatorname{PGL}(V)$, these flags satisfy
\[
 \mathfrak f^\pm(g^m)=\mathfrak f^\pm(g),
 \qquad
 \mathfrak f^\pm(hgh^{-1})=h\mathfrak f^\pm(g).
\]
For a nontrivial element $g$ of a free group $F$, we write
$g^+,g^-\in\partial F$ for its attracting and repelling points in the
Gromov boundary.  These boundary points are distinct from the projective
flags $\mathfrak f^\pm(g)$.

A \emph{ping-pong pair} $a,b$ consists of very proximal
transformations with common parameters $r>2\varepsilon>0$ such that
$d(v_h,H_g)\ge r$ whenever
$g,h\in\{a,a^{-1},b,b^{-1}\}$ and $h\ne g^{-1}$.
Thus the attracting $\varepsilon$-neighborhood of $h$ avoids the
repelling $\varepsilon$-neighborhood of $g$, and $g$ sends it into its own
attracting neighborhood.  Requiring these conditions in both
$\mathbb P(V)$ and $\mathbb P(V^*)$ gives a \emph{projective Schottky group}
$\langle a,b\rangle$, freely generated by $a,b$ by the ping-pong lemma.
All locally uniform convergence below
is uniform on compact subsets of the indicated open domain.

We first isolate the algebraic and analytic input taken from \cite{BG-topological}.

\begin{proposition}[Projective Tits input]
\label{prop:projective-input}
Let $K$ be a field finitely generated over its prime field, let $\mathbf H$
be a linear algebraic $K$-group such that $\mathbf H^\circ$ is not
solvable, and let $\Delta\le\mathbf H(K)$ be finitely generated and Zariski
dense.  Put $\Delta^\circ=\Delta\cap\mathbf H^\circ$.  Then there exist
\begin{enumerate}[label=(\roman*)]
\item a local field $k$ and an embedding $K\hookrightarrow k$;
\item a finite-dimensional $k$-vector space $V$ and an algebraic
      projective representation
      \[
        \rho:\mathbf H(k)\longrightarrow\operatorname{PGL}(V);
      \]
\item an $\mathbf H$-invariant closed projective subvariety over $k$
      \[
        \mathcal X\subset\mathcal I(V);
      \]
\item $r>0$;            
\end{enumerate}
such that the projective action $\rho(\mathbf H(k))$ {on $\mathbb P(V)$} is strongly
irreducible,  
and for every sufficiently small
$\varepsilon>0$, the group $\Delta^\circ$ contains $\delta$ such that
$\rho(\delta)$ is $(r,\varepsilon)$-very proximal and
\[
 \mathfrak f^+(\rho(\delta)),\mathfrak f^-(\rho(\delta))\in\mathcal X.
\]
\end{proposition}

\begin{proof}
Theorem~4.3 of \cite{BG-topological}, in the form recorded as
\cite[Theorem~7.6]{GG}, gives a local
field $k$, an embedding $K\hookrightarrow k$, and a strongly irreducible
projective representation
\[
 \rho:\mathbf H(k)\longrightarrow\operatorname{PGL}(V)
\]
with the property that
$\Delta^\circ$ contains $(r,\varepsilon)$-very proximal elements for a
fixed $r>0$ and arbitrarily small $\varepsilon>0$.

We use the geometric structure supplied by the highest-weight construction
in \cite[Propositions~4.1--4.2 and the proof of Theorem~4.3]{BG-topological}.
The solvable radical acts projectively trivially, and the restriction to
$\mathbf H^\circ$ is the projectivization of an absolutely irreducible
highest-weight representation $\sigma$ factoring through a nontrivial
semisimple quotient.  Its highest weight is chosen to be invariant under
the action of the component group $\mathbf H/\mathbf H^\circ$ on dominant
weights.  This action is induced by conjugation, followed by an inner
adjustment preserving a fixed Borel subgroup and maximal torus.
Consequently, for every $h\in\mathbf H(k)$, the representations $\sigma$
and $x\mapsto\sigma(hxh^{-1})$ are equivalent.  Choose an intertwiner
$A_h\in\operatorname{GL}(V)$ satisfying
\[
 A_h\sigma(x)A_h^{-1}=\sigma(hxh^{-1})
 \qquad(x\in\mathbf H^\circ(k)).
\]
By Schur's lemma, $A_h$ is unique up to scalar, and $A_hA_g$ differs from
$A_{hg}$ by a scalar.  Thus $h\mapsto[A_h]$ defines a projective
representation of $\mathbf H(k)$ extending the projectivization of
$\sigma$, as in \cite[Proposition~4.2]{BG-topological}.

Let
$X^+\subset\mathbb P(V)$ and $X^-\subset\mathbb P(V^*)$ be the corresponding
closed highest-weight orbits.  The closed orbit of incident
line--hyperplane pairs in $X^+\times X^-$ is a partial flag variety
$\mathcal X\subset\mathcal I(V)$, invariant under $\mathbf H$.
The canonical attracting points and repelling hyperplanes of the very
proximal elements furnished by this construction lie in the corresponding
closed orbits; their associated flags belong to $\mathcal X$.
Since $\mathcal X$ is projective over the local field $k$, its set of
$k$-points is compact.
\end{proof}

We use $\mathcal X$ also for its compact space $\mathcal X(k)$ of
$k$-points.  

\medskip

\paragraph{\bf Opposition on $\mathcal X$.}
If $x=(\ell,H)$ and $y=(\ell',H')$ are in $\mathcal X$, define
\begin{equation}\label{eq:projective-opposition}
 x\pitchfork y
 \quad\Longleftrightarrow\quad
 \ell\notin H'\ \text{ and }\ \ell'\notin H,
 \qquad
 \Opp(x)=\{y\in\mathcal X:y\pitchfork x\}.
\end{equation}
For the highest-weight orbits occurring in the proof of
Proposition~\ref{prop:projective-input}, the relation $\pitchfork$ is
symmetric, $\mathbf H(k)$-invariant, and Zariski open, and
\[
 \mathcal O=\{(x,y)\in\mathcal X^2:x\pitchfork y\}
\]
is the unique big open geometric $\mathbf H^\circ$-orbit in
$\mathcal X\times\mathcal X$.  In particular, $\mathcal X$ and
$\mathcal O$ are irreducible, and $\Opp(x)$ is a nonempty dense
Zariski-open subset of $\mathcal X$ for every $x\in\mathcal X$.  The two
flags associated with each very proximal element supplied by
Proposition~\ref{prop:projective-input} form an opposite pair.

We henceforth suppress $\rho$ from the notation for the action of
$\mathbf H(k)$ on $\mathcal X$; in particular,
$\mathfrak f^\pm(g)$ abbreviates $\mathfrak f^\pm(\rho(g))$.

\begin{lemma}[Simultaneous bridges]
\label{lem:simultaneous-bridges}
Keep the hypotheses and notation of Proposition~\ref{prop:projective-input},
and let $E\subset\Delta$ be finite.  There is a very proximal element
$a\in\Delta^\circ$ whose associated flags
$\mathfrak f^+(a),\mathfrak f^-(a)\in\mathcal X$ satisfy
\begin{equation}\label{eq:seed-bridges}
 \mathfrak f^+(a)\pitchfork s\mathfrak f^-(a),
 \qquad s\in E.
\end{equation}
\end{lemma}

\begin{proof}
Choose a very proximal element $a_0\in\Delta^\circ$, and put
\[
 x_0=\mathfrak f^+(a_0),
 \qquad
 y_0=\mathfrak f^-(a_0).
\]
Thus $(x_0,y_0)\in\mathcal O$.  Set
\[
 \mathcal W=
 \{(x,y)\in\mathcal O:x\pitchfork sy\text{ for every }s\in E\}.
\]
This is Zariski open in $\mathcal O$.  It is nonempty: for any
$y\in\mathcal X$, the finite intersection
\[
 \Opp(y)\cap\bigcap_{s\in E}\Opp(sy)
\]
is a nonempty open subset of the irreducible variety $\mathcal X$; choose
$x$ in this intersection.  Then $(x,y)\in\mathcal W$.

Consider the orbit morphism
\[
 \Phi_{a_0}:\mathbf H^\circ\longrightarrow\mathcal O,
 \qquad h\longmapsto(hx_0,hy_0).
\]
Let $\bar k$ be an algebraic closure of $k$.  Since $\mathcal O_{\bar k}$
is a single $\mathbf H^\circ_{\bar k}$-orbit containing $(x_0,y_0)$,
the orbit morphism $\Phi_{a_0}$ is surjective after base change to
$\bar k$.  Hence $\Phi_{a_0}^{-1}(\mathcal W)$ is a nonempty Zariski-open
subset of $\mathbf H^\circ$.  Since $\Delta^\circ$ is Zariski dense in $\mathbf H^\circ$, we may
choose $h\in\Delta^\circ$ such that
$(hx_0,hy_0)\in\mathcal W$.  The conjugate $a=ha_0h^{-1}$ has the required
properties, since
$\mathfrak f^\pm(ha_0h^{-1})=h\mathfrak f^\pm(a_0)$.
\end{proof}

\begin{lemma}[Schottky completion]
\label{lem:schottky-completion}
Let $a\in\Delta^\circ$ be as in
Lemma~\ref{lem:simultaneous-bridges}.  There are $g\in\Delta^\circ$ and
$L\in\mathbb{N}$ such that
\[
 t=a^L,
 \qquad
 b=ga^Lg^{-1}
\]
form a ping-pong pair of very proximal transformations simultaneously in
$\mathbb P(V)$ and $\mathbb P(V^*)$.  In particular,
$B=\langle t,b\rangle$ is a rank-two projective Schottky group.
\end{lemma}

\begin{proof}
Put
\[
 x=\mathfrak f^+(a),
 \qquad
 y=\mathfrak f^-(a),
 \qquad
 A=\Opp(x)\cap\Opp(y).
\]
The set
$(A\times A)\cap\mathcal O$ is a nonempty Zariski-open subset of
$\mathcal O$.  Applying the same orbit map and using the density of
$\Delta^\circ$, choose $g\in\Delta^\circ$ such that
\[
 (gx,gy)\in(A\times A)\cap\mathcal O.
\]
Thus the four flags $x,y,gx,gy$ are pairwise opposite.  For a
sufficiently large integer $L$, the elements
\[
 t=a^L,
 \qquad
 b=ga^Lg^{-1}
\]
form a ping-pong pair of very proximal transformations simultaneously in
$\mathbb P(V)$ and $\mathbb P(V^*)$.  Hence $B=\langle t,b\rangle$ is a
rank-two projective Schottky group.
\end{proof}

\begin{lemma}[Schottky boundary dynamics]
\label{lem:schottky-boundary}
Let $B=\langle t,b\rangle$ be the projective Schottky group supplied by
Lemma~\ref{lem:schottky-completion}.  There is a continuous
$B$-equivariant embedding
\[
 \zeta:\partial B\longrightarrow\mathcal X
\]
with the following properties:
\begin{enumerate}[label=(\alph*)]
\item $\zeta(\xi)\pitchfork\zeta(\eta)$ whenever $\xi\ne\eta$;
\item for every nontrivial $w\in B$,
      \[
       \zeta(w^+)=\mathfrak f^+(w),
       \qquad
       \zeta(w^-)=\mathfrak f^-(w);
      \]
\item if $w_n\to\xi$ and $w_n^{-1}\to\eta$ in the Gromov compactification
      of $B$, then
      \[
        w_n\big|_{\Opp(\zeta(\eta))}\longrightarrow\zeta(\xi)
      \]
      locally uniformly.
\end{enumerate}
\end{lemma}

\begin{proof}
The usual nested-domain construction gives a continuous equivariant
boundary embedding $\zeta:\partial B\to\mathcal X$, pairwise opposition of
distinct boundary flags, and the asserted locally uniform convergence on
opposition cells; see \cite[Section~3]{BG-topological}.  The same
construction identifies the images of the fixed points of every
nontrivial $w\in B$ with the two flags associated with $w$.
\end{proof}

Apply Proposition~\ref{prop:projective-input} with $\Delta=\Gamma$, then
Lemma~\ref{lem:simultaneous-bridges} with
$E=\{s_1,\ldots,s_d\}$, and finally
Lemmas~\ref{lem:schottky-completion}--\ref{lem:schottky-boundary}.  Fix the
resulting action, flag variety, Schottky pair $t,b$, and boundary embedding
$\zeta$ for the rest of the paper.  By
Lemma~\ref{lem:schottky-boundary}\textup{(b)},
\[
 \zeta(t^\pm)=\mathfrak f^\pm(t)=\mathfrak f^\pm(a),
\]
because taking a positive power does not change the canonical attracting
points or repelling hyperplanes.  Thus \eqref{eq:seed-bridges} becomes
\begin{equation}\label{eq:bridge-conditions}
 \zeta(t^+)\pitchfork s_i\zeta(t^-),
 \qquad 1\le i\le d.
\end{equation}

\section{The quasi-regular estimate}\label{sec:window}

The final estimate is group theoretic.  It is the finite-family form of
{Yao's} quasi-regular lemma \cite[Lemma~2.1]{Yao}.

\begin{lemma}[Quasi-regular estimate]\label{lem:window}
Let $\Lambda=\langle t,Y\rangle$, where $Y$ is finite.  For $M\ge0$ put
\[
 H_M=\bigl\langle t^jyt^{-j}:y\in Y,\ 0\le j\le M\bigr\rangle.
\]
If $[\Lambda:H_M]=\infty$, then
\[
 \kk(\Lambda,\{t\}\cup Y)\le\sqrt{\frac{2}{M+1}}.
\]
\end{lemma}

\begin{proof}
We first show that the right cosets
\begin{equation}\label{eq:window-cosets}
 H_Mt^{-1},H_M,H_Mt,\ldots,H_Mt^M
\end{equation}
are pairwise distinct.  Suppose otherwise that
$H_Mt^a=H_Mt^b$ for some $-1\le a<b\le M$.  Then
$t^r\in H_M$, where $r=b-a$ and $1\le r\le M+1$.

For every $y\in Y$ and $0\le i<r$, we have $i\le M$, and hence
$t^iyt^{-i}\in H_M$ by the definition of $H_M$.  Given any
$k\in\mathbb Z$, write $k=qr+i$ with $q\in\mathbb Z$ and $0\le i<r$.
Since $t^r\in H_M$, conjugation by $(t^r)^q$ preserves $H_M$, and therefore
\[
 t^kyt^{-k}
 =(t^r)^q(t^iyt^{-i})(t^r)^{-q}
 \in H_M.
\]
Let
\[
 N=\langle t^kyt^{-k}:k\in\mathbb Z,\ y\in Y\rangle.
\]
Then $N\triangleleft\Lambda$, the quotient $\Lambda/N$ is cyclic and
generated by the image of $t$, and both $N$ and $t^r$ are contained in
$H_M$.  Consequently
\[
 [\Lambda:H_M]
 \le [\Lambda:\langle N,t^r\rangle]
 \le r,
\]
contrary to the hypothesis that $H_M$ has infinite index.

Now use the right quasi-regular representation on
$\ell^2(H_M\backslash\Lambda)$,
\[
 \pi(g)\delta_{H_Mx}=\delta_{H_Mxg^{-1}}.
\]
It has no nonzero invariant vector because the transitive set
$H_M\backslash\Lambda$ is infinite.  Put
\[
 \xi_M=\frac1{\sqrt{M+1}}
       \sum_{j=0}^{M}\delta_{H_Mt^j}.
\]
If $y\in Y$ and $0\le j\le M$, then
\[
 t^jy^{-1}=(t^jy^{-1}t^{-j})t^j
\]
and $t^jy^{-1}t^{-j}\in H_M$.  Thus
$H_Mt^jy^{-1}=H_Mt^j$, so every $y\in Y$ fixes each summand of $\xi_M$.
On the other hand, \eqref{eq:window-cosets} gives
\[
 \|\pi(t)\xi_M-\xi_M\|=
 \|\frac1{\sqrt{M+1}}
 \bigl(\delta_{H_Mt^{-1}}-\delta_{H_Mt^M}\bigr)\|=\sqrt{\frac{2}{M+1}}.
\]
\end{proof}

\section{Schreier coordinates and boundary separation}\label{sec:tail}
Let $\langle t,b\rangle$ be a free group of rank $2$.\footnote{Recall the Schottky elements $t,b$ constructed in the last paragraph of \S 2.}
Write
\[
 b_j=t^jbt^{-j},\qquad j\in\mathbb Z.
\]
The kernel of the homomorphism
\[
 B=\langle t,b\rangle\longrightarrow\mathbb Z,
 \qquad t\longmapsto1,\quad b\longmapsto0,
\]
is freely based by $\{b_j:j\in\mathbb Z\}$; see
\cite[Chapter~I, Proposition~3.7]{LS}.  For $M\ge1$, put
\[
 J_M=\langle b_0,\ldots,b_M\rangle.
\]

\begin{lemma}[Free-factor boundary separation]
\label{lem:boundary-separation}
Let $I\subset\mathbb Z$ be finite and disjoint from
$\{0,\ldots,M\}$, and put
\[
 B_I=\langle b_j:j\in I\rangle.
\]
Then
\[
 \langle J_M,B_I\rangle=J_M*B_I.
\]
In $\partial B$, every point of $\partial B_I$ lies outside
$\partial J_M$ and has trivial stabilizer in $J_M$.  Moreover, distinct
points of $\partial B_I$ have disjoint $J_M$-orbits.
\end{lemma}

\begin{proof}
The free-product identity follows because $J_M$ and $B_I$ are generated
by disjoint subsets of the free basis $\{b_j:j\in\mathbb Z\}$.

Set $L=\langle J_M,B_I\rangle=J_M*B_I$.  Since $L$ is a finitely
generated subgroup of the free group $B$, it is quasiconvex, and hence
$\partial L$ embeds $L$-equivariantly in $\partial B$.  Reduced normal
forms in the free product $J_M*B_I$ show that the subsets
\[
 h\partial B_I,\qquad h\in J_M,
\]
are pairwise disjoint and disjoint from $\partial J_M$.  The three
assertions follow.
\end{proof}

\section{Relative ping-pong}\label{sec:pingpong}

{The following is the finite-family version of the relative ping-pong
criterion from \cite{Yao} (cf.\ \cite[Proposition~3.5]{Yao}).}  We include the argument because we
work over an arbitrary local field and allow several added cyclic factors.

\medskip

\paragraph{\bf $J$-separated families.}
Let $F$ be a free group, let $J<F$ be quasiconvex, and let $A$ be
finite.  A family
$((p_a,q_a))_{a\in A}$ in $(\partial F\setminus\partial J)^2$ is called
\emph{$J$-separated} if every $p_a$ and $q_a$ has trivial stabilizer in
$J$, and if the $2|A|$ orbits $Jp_a,Jq_a$ are pairwise disjoint.

\begin{proposition}[Finite-family relative ping-pong]
\label{prop:relative-pp}
Let $L$ be a group acting by homeomorphisms on $\mathcal X$, and let $F<L$
be a non-elementary projective Schottky group with a boundary embedding
\[
 \zeta:\partial F\longrightarrow\mathcal X
\]
satisfying properties \textup{(a)}--\textup{(c)} of
Lemma~\ref{lem:schottky-boundary}.  Let $J<F$ be a non-elementary
quasiconvex subgroup, and let $A$ be a finite set.

Let $((p_a,q_a))_{a\in A}$ be a $J$-separated family; we call $p_a$ and
$q_a$ respectively the limiting target and source.  For each $a\in A$,
let $(g_{a,N})_{N\ge1}$ be a sequence in $L$, not necessarily in $F$,
such that, as $N\to\infty$,
\begin{equation}\label{eq:dynamic-convergence}
 \begin{aligned}
 g_{a,N}x&\longrightarrow\zeta(p_a)
   &&\text{locally uniformly for }x\in\Opp(\zeta(q_a)),\\
 g_{a,N}^{-1}x&\longrightarrow\zeta(q_a)
   &&\text{locally uniformly for }x\in\Opp(\zeta(p_a)).
 \end{aligned}
\end{equation}
Then, for all sufficiently large $N$, every $g_{a,N}$ has infinite order
and the natural homomorphism
\[
 J*\mathop{*}_{a\in A}\langle g_{a,N}\rangle
 \longrightarrow
 \bigl\langle J,g_{a,N}:a\in A\bigr\rangle
\]
is an isomorphism.
\end{proposition}

\begin{proof}
The action of $J$ on $\partial F\setminus\partial J$ is properly
discontinuous.  Put
\[
 P_0=\{p_a,q_a:a\in A\},
 \qquad
 C_0=\partial J\cup
     \overline{(JP_0)\setminus P_0}^{\,\partial F}.
\]
The orbit and stabilizer assumptions imply that $C_0$ is compact and
disjoint from $P_0$.  Choose a compact neighborhood $D$ of $\zeta(C_0)$
and pairwise disjoint compact neighborhoods $E_a^+,E_a^-$ of
$\zeta(p_a),\zeta(q_a)$ such that $D$ is mutually opposite to every
one of these neighborhoods, and any two distinct neighborhoods among the
$E_a^\pm$ are mutually opposite.

After shrinking the neighborhoods $E_a^\pm$,
\begin{equation}\label{eq:J-ping}
 h\Bigl(\bigcup_{a\in A}(E_a^+\cup E_a^-)\Bigr)
 \subset\Int(D),\qquad 1\ne h\in J.
\end{equation}
Indeed, a failure gives either a bounded sequence in $J$, contradicting
the orbit and stabilizer conditions, or an escaping sequence $(h_n)$.
In the latter case, after passing to a subsequence, we have $h_n\to\xi$
and $h_n^{-1}\to\eta$ in the Gromov compactification of $F$, for some
$\xi,\eta\in\partial J$.  Since $\zeta(\eta)\in D$ and $D$ is mutually
opposite to every $E_a^\pm$, the union of these neighborhoods is a compact
subset of $\Opp(\zeta(\eta))$.  Lemma~\ref{lem:schottky-boundary}\textup{(c)}
therefore gives $h_nx\to\zeta(\xi)$ uniformly on this union.
As $\zeta(\xi)\in\zeta(\partial J)\subset\Int(D)$, the inclusion in
\eqref{eq:J-ping} holds for $h_n$ whenever $n$ is sufficiently large,
a contradiction.

For each $a\in A$, set $E_{p_a}=E_a^+$ and $E_{q_a}=E_a^-$.  
By \eqref{eq:dynamic-convergence}, for sufficiently
large $N$,
\begin{equation}\label{eq:g-ping}
 \begin{aligned}
 g_{a,N}\Bigl(D\cup
   \bigcup_{r\in P_0\setminus\{q_a\}}E_r\Bigr)&\subset\Int(E_a^+),\\
 g_{a,N}^{-1}\Bigl(D\cup
   \bigcup_{r\in P_0\setminus\{p_a\}}E_r\Bigr)&\subset\Int(E_a^-).
 \end{aligned}
\end{equation}
Iteration gives the same inclusions for every nonzero power, with the sign
selecting $E_a^+$ or $E_a^-$.  Equations
\eqref{eq:J-ping}--\eqref{eq:g-ping} are the ordinary multi-factor
ping-pong hypotheses.  Every nontrivial reduced word acts nontrivially,
proving the free-product conclusion and the infinitude of the cyclic
factors.
\end{proof}

Recall that $B=\langle t,b\rangle$ where $t,b$ are the Schottky elements constructed in \S 2.

\begin{lemma}[Sandwich dynamics]\label{lem:sandwich}
Let $u,v\in B$ be nontrivial and let $s\in\mathbf H(k)$.  If
\begin{equation}\label{eq:sandwich-transversality}
 \zeta(u^-)\pitchfork s\zeta(v^+),
\end{equation}
then, for $x_N=u^Nsv^N$, locally uniformly,
\[
 \begin{aligned}
 x_N\big|_{\Opp(\zeta(v^-))}&\longrightarrow\zeta(u^+),\\
 x_N^{-1}\big|_{\Opp(\zeta(u^+))}&\longrightarrow\zeta(v^-).
 \end{aligned}
\]
\end{lemma}

\begin{proof}
On compact subsets of $\Opp(\zeta(v^-))$, the maps $v^N$ converge to
$\zeta(v^+)$.  After applying $s$, condition
\eqref{eq:sandwich-transversality} puts the image in a compact subset of
$\Opp(\zeta(u^-))$, where $u^N$ converges to $\zeta(u^+)$.  The same
reasoning applies to $x_N^{-1}=v^{-N}s^{-1}u^{-N}$.  
\end{proof}

\section{Two-sided Schreier corrections}\label{sec:corrections}

\begin{lemma}[Correction blocks]
\label{lem:correction-blocks}
For every $M\ge1$, there are positive integers
\[
 r_1,\ldots,r_d,\qquad q_1,\ldots,q_d
\]
such that, on putting
\[
 u_i=b_{r_i},\qquad v_i=b_{-q_i},
\]
the following properties hold.
\begin{enumerate}[label=(\roman*)]
\item
\begin{equation}\label{eq:two-sided-opposition}
 \zeta(u_i^-)\pitchfork s_i\zeta(v_i^+),
 \qquad 1\le i\le d.
\end{equation}
\item The $2d$ integer intervals
\[
 [r_i,r_i+M],
 \qquad
 [-q_i,M-q_i],
 \qquad 1\le i\le d,
\]
are pairwise disjoint and are disjoint from $[0,M]$.
\end{enumerate}
\end{lemma}

\begin{proof}
Since $t,b$ form a Schottky pair, $b^-\ne t^-$ and $b^+\ne t^+$.  Hence,
in $\partial B$,
\[
 t^rb^-\longrightarrow t^+,
 \qquad
 t^{-q}b^+\longrightarrow t^-
 \qquad\text{as }r,q\longrightarrow\infty.
\]
By continuity and equivariance of $\zeta$,
\[
 \zeta(b_r^-)\longrightarrow\zeta(t^+),
 \qquad
 \zeta(b_{-q}^+)\longrightarrow\zeta(t^-).
\]
Condition \eqref{eq:bridge-conditions} together with openness of opposition
gives that, for every $i$
\[
 \zeta(b_r^-)\pitchfork s_i\zeta(b_{-q}^+)
\]
whenever $r$ and $q$ are sufficiently large.  We may choose the $r_i$ and
$q_i$ successively, as large as necessary, so that
\eqref{eq:two-sided-opposition} holds and all the indicated intervals are
pairwise disjoint and avoid $[0,M]$.
\end{proof}

Fix choices supplied by Lemma~\ref{lem:correction-blocks}.  For
$1\le i\le d$ and $0\le j\le M$, define
\begin{equation}\label{eq:two-sided-endpoints}
 \begin{aligned}
 p_{ij}&=t^ju_i^+=b_{j+r_i}^+,\\
 q_{ij}&=t^jv_i^-=b_{j-q_i}^-.
 \end{aligned}
\end{equation}
Let
\[
 I_M=
 \bigcup_{i=1}^d
 \bigl(\{r_i,\ldots,r_i+M\}
       \cup\{-q_i,\ldots,M-q_i\}\bigr)
\]
and put
\[
 B_0=\langle b_k:k\in I_M\rangle.
\]

\begin{lemma}[Endpoint separation]
\label{lem:endpoint-separation}
The $2d(M+1)$ points in \eqref{eq:two-sided-endpoints} are pairwise
distinct and belong to $\partial B_0$.  Moreover, the family
\[
 ((p_{ij},q_{ij}))_{1\le i\le d,\,0\le j\le M}
\]
is $J_M$-separated in $\partial B$.
\end{lemma}

\begin{proof}
For every $n\in\mathbb Z$, equivariance gives $b_n^\pm=t^nb^\pm$, so the
identities in \eqref{eq:two-sided-endpoints} hold.  Distinct members of the
Schreier basis $\{b_n:n\in\mathbb Z\}$ generate noncommensurable maximal
cyclic subgroups of the free group $B$.  Their fixed-point sets in
$\partial B$ are therefore disjoint.  The disjointness of the index blocks
proves that the displayed points are pairwise distinct.  Their indices all
belong to $I_M$, so the points lie in $\partial B_0$.

By construction, $I_M\cap\{0,\ldots,M\}=\varnothing$.  Lemma
\ref{lem:boundary-separation} now shows that the displayed points lie
outside $\partial J_M$, have trivial $J_M$-stabilizers, and have pairwise
disjoint $J_M$-orbits.  Thus the family is $J_M$-separated.
\end{proof}

Define
\begin{equation}\label{eq:corrected-generators}
 x_i(N)=u_i^Ns_iv_i^N,
 \qquad 1\le i\le d.
\end{equation}
For $0\le j\le M$, put
\[
 g_{ij,N}=t^jx_i(N)t^{-j}.
\]
Lemma~\ref{lem:sandwich} and equivariance give, as $N\to\infty$,
\begin{equation}\label{eq:window-dynamics}
 \begin{aligned}
 g_{ij,N}\big|_{\Opp(\zeta(q_{ij}))}
   &\longrightarrow\zeta(p_{ij}),\\
 g_{ij,N}^{-1}\big|_{\Opp(\zeta(p_{ij}))}
   &\longrightarrow\zeta(q_{ij}),
 \end{aligned}
\end{equation}
locally uniformly.

\begin{corollary}[A free family of conjugates]\label{cor:free-window}
For all sufficiently large $N$,
\[
 H_{M,N}=
 \left\langle
 t^jbt^{-j},\ t^jx_i(N)t^{-j}:
 0\le j\le M,\ 1\le i\le d
 \right\rangle
\]
is free on the displayed $(M+1)(d+1)$ elements.
\end{corollary}

\begin{proof}
The group $J_M$ is free on $b_0,\ldots,b_M$, non-elementary, and
quasiconvex in $B$.  By Lemma~\ref{lem:endpoint-separation}, the family
$((p_{ij},q_{ij}))$ is $J_M$-separated.  
Proposition~\ref{prop:relative-pp} apply to the family indexed by $(i,j)$, together with \eqref{eq:window-dynamics} gives
\[
 H_{M,N}\cong
 J_M*\mathop{*}_{\substack{1\le i\le d\\0\le j\le M}}
 \langle g_{ij,N}\rangle.
\]
\end{proof} 

\section{Proof of the theorem}\label{sec:proof}

\begin{proof}[Proof of Theorem~\ref{thm:main}]
As observed in the introduction, we may assume that $\Gamma$ has
property~$(T)$.  Fix $M\ge1$, make the choices in
Lemma~\ref{lem:correction-blocks}, and take $N=N(M)$ large enough for
Corollary~\ref{cor:free-window}.  Put
\[
 S_M=\{t,b,x_1(N),\ldots,x_d(N)\}\subset\Gamma.
\]
Equation \eqref{eq:corrected-generators} recovers every marked generator:
\[
 s_i=u_i^{-N}x_i(N)v_i^{-N}.
\]
Thus $S_M$ generates $\Gamma$ and $|S_M|\le d+2$.

The subgroup $H_{M,N}$ is a nonabelian free group.  In particular, it has infinite
index in $\Gamma$.
Apply
Lemma~\ref{lem:window} with
\[
 Y=\{b,x_1(N),\ldots,x_d(N)\}.
\]
We obtain
\[
 \kk(\Gamma,S_M)\le\sqrt{\frac{2}{M+1}}.
\]
Letting $M$ tend to infinity proves the theorem.
\end{proof}

{
\section{A nonlinear extension: convergence groups}
\label{sec:convergence-extension}

The linear structure is used above only to produce the dynamical package
at the end of Section~\ref{sec:algebraic}: a Schottky subgroup, a compact
boundary space with an opposition relation, and the simultaneous bridge
conditions.  Once that package is available, the remaining argument is
topological and group theoretic.  This observation gives an extension to
groups which need not be linear.

Recall that an action of a group on a compact metrizable space $X$ is a
\emph{convergence action} if every sequence of distinct group elements has
a subsequence $(g_n)$ for which there are points $a,r\in X$ such that
\[
   g_n\big|_{X\setminus\{r\}}\longrightarrow a
\]
locally uniformly.  Thus, for every compact
$C\subseteq X\setminus\{r\}$ and every neighborhood $U$ of $a$, one has
$g_nC\subseteq U$ for all sufficiently large $n$.  The action is
\emph{non-elementary} if its limit set $\Lambda$ has more than two points.
After restricting the action to $\Lambda$, it is again a convergence
action, and we work on $\Lambda$ from now on.  We use the standard facts
that $\Lambda$ is then infinite, perfect, and
minimal, and that the ordered fixed-point pairs of loxodromic elements are
dense in
\[
   \Lambda^{(2)}=\{(x,y)\in\Lambda^2:x\ne y\}.
\]
We also use the standard high-power construction of Schottky subgroups of
convergence groups.  These facts may be found in
\cite{Tukia,Bowditch-config}; see also \cite[Section~4]{GG}.  In
particular, the density assertion is \cite[Theorem~2R]{Tukia}.

\begin{lemma}[The convergence-group Schottky package]
\label{lem:convergence-schottky}
Let $\Gamma$ admit a non-elementary convergence action with limit set
$\Lambda$, and let $E\subseteq\Gamma$ be finite.  There are elements
$t,b\in\Gamma$ which generate a rank-two free group
$B=\langle t,b\rangle$ and a continuous $B$-equivariant embedding
\[
   \zeta:\partial B\longrightarrow\Lambda
\]
with the following properties:
\begin{enumerate}[label=(\alph*)]
\item $\zeta(\xi)\ne\zeta(\eta)$ whenever $\xi\ne\eta$;
\item for every nontrivial $w\in B$, the points $\zeta(w^+)$ and
      $\zeta(w^-)$ are respectively the attracting and repelling fixed
      points of $w$ on $\Lambda$;
\item if $w_n\to\xi$ and $w_n^{-1}\to\eta$ in the Gromov compactification
      of $B$, then
      \[
         w_n\big|_{\Lambda\setminus\{\zeta(\eta)\}}
         \longrightarrow\zeta(\xi)
      \]
      locally uniformly;
\item $\zeta(t^+)\ne s\zeta(t^-)$ for every $s\in E$.
\end{enumerate}
\end{lemma}

\begin{proof}
Choose $y\in\Lambda$ and
\[
   x\in\Lambda\setminus
   \bigl(\{y\}\cup\{sy:s\in E\}\bigr).
\]
There are neighborhoods $U$ of $x$ and $V$ of $y$ such that
\[
   U\cap V=\varnothing,
   \qquad
   U\cap sV=\varnothing
   \quad(s\in E).
\]
By the density of loxodromic fixed-point pairs, there is a loxodromic
element $a\in\Gamma$ such that
\[
   a^+\in U,
   \qquad
   a^-\in V.
\]
In particular,
\begin{equation}\label{eq:convergence-bridges}
   a^+\ne sa^-,
   \qquad s\in E.
\end{equation}

Choose a loxodromic element $c$ independent from $a$.  By the usual
convergence-group Schottky construction, for some $\ell\ge1$ the elements
\[
   t=a^\ell,
   \qquad
   b=c^\ell
\]
generate a rank-two Schottky group.  The construction gives a continuous
$B$-equivariant boundary embedding
$\zeta:\partial B\to\Lambda$, and gives properties~\textup{(a)} and
\textup{(b)}.  Moreover,
\[
   \zeta(t^+)=a^+,
   \qquad
   \zeta(t^-)=a^-,
\]
so \eqref{eq:convergence-bridges} proves~\textup{(d)}.

We verify~\textup{(c)}.  Let
\[
   w_n\to\xi,
   \qquad
   w_n^{-1}\to\eta.
\]
The word lengths of the $w_n$ tend to infinity.  Every subsequence
therefore has a further subsequence of distinct terms which is a
convergence sequence on $\Lambda$; denote its attracting and repelling
points by $p$ and $q$.  Choose
$\theta\in\partial B\setminus\{\eta\}$ such that
$\zeta(\theta)\ne q$.  Convergence on the free-group boundary and
equivariance give
\[
   w_n\zeta(\theta)=\zeta(w_n\theta)
   \longrightarrow\zeta(\xi),
\]
whereas the convergence property gives
$w_n\zeta(\theta)\to p$.  Hence $p=\zeta(\xi)$.  Applying the same
argument to $(w_n^{-1})$ gives $q=\zeta(\eta)$: here we use the standard
fact that if a convergence sequence collapses to $p$ off $q$, then its
inverse sequence collapses to $q$ off $p$.

Thus every subsequence has a further subsequence which converges locally
uniformly to $\zeta(\xi)$ off $\zeta(\eta)$.  If the full sequence did not
have this property, a compact subset of
$\Lambda\setminus\{\zeta(\eta)\}$, a neighborhood of $\zeta(\xi)$, and a
further subsequence would contradict the preceding conclusion.  This
proves~\textup{(c)}.
\end{proof}

\begin{proposition}[Relative ping-pong for convergence actions]
\label{prop:convergence-relative-pp}
Let $L$ act by homeomorphisms on a compact metrizable space $\Lambda$,
and let $F<L$ be a rank-two free group admitting a continuous
$F$-equivariant embedding
\[
   \zeta:\partial F\longrightarrow\Lambda
\]
with properties~\textup{(a)} and~\textup{(c)} of
Lemma~\ref{lem:convergence-schottky}.  Let $J<F$ be a non-elementary
quasiconvex subgroup, let $A$ be finite, and let
$((p_a,q_a))_{a\in A}$ be a $J$-separated family.  Suppose that
$g_{a,N}\in L$ and, as $N\to\infty$,
\begin{equation}\label{eq:convergence-relative-dynamics}
 \begin{aligned}
 g_{a,N}x&\longrightarrow\zeta(p_a)
   &&\text{locally uniformly on }\Lambda\setminus\{\zeta(q_a)\},\\
 g_{a,N}^{-1}x&\longrightarrow\zeta(q_a)
   &&\text{locally uniformly on }\Lambda\setminus\{\zeta(p_a)\}.
 \end{aligned}
\end{equation}
Then, for all sufficiently large $N$, every $g_{a,N}$ has infinite order
and the natural homomorphism
\[
   J*\mathop{*}_{a\in A}\langle g_{a,N}\rangle
   \longrightarrow
   \langle J,g_{a,N}:a\in A\rangle
\]
is an isomorphism.
\end{proposition}

\begin{proof}
Define $z\pitchfork w$ to mean $z\ne w$, so that
$\Opp(z)=\Lambda\setminus\{z\}$.  With this choice, the proof of
Proposition~\ref{prop:relative-pp} applies verbatim.  Indeed, compactness
and metrizability give the disjoint compact neighborhoods used there;
property~\textup{(c)} of Lemma~\ref{lem:convergence-schottky} controls
escaping sequences in $J$; and
\eqref{eq:convergence-relative-dynamics} supplies the remaining
ping-pong inclusions.
\end{proof}

\begin{lemma}[Sandwich dynamics for convergence actions]
\label{lem:convergence-sandwich}
Let $B<\Gamma$ and $\zeta:\partial B\to\Lambda$ be as in
Lemma~\ref{lem:convergence-schottky}.  If $u,v\in B$ are nontrivial,
$s\in\Gamma$, and
\[
   \zeta(u^-)\ne s\zeta(v^+),
\]
then, for $x_N=u^Nsv^N$, locally uniformly,
\[
 \begin{aligned}
 x_N\big|_{\Lambda\setminus\{\zeta(v^-)\}}
    &\longrightarrow\zeta(u^+),\\
 x_N^{-1}\big|_{\Lambda\setminus\{\zeta(u^+)\}}
    &\longrightarrow\zeta(v^-).
 \end{aligned}
\]
\end{lemma}

\begin{proof}
On every compact subset of
$\Lambda\setminus\{\zeta(v^-)\}$, the maps $v^N$ converge uniformly to
$\zeta(v^+)$.  After applying $s$, the images are eventually contained
in a compact set disjoint from $\zeta(u^-)$, and $u^N$ converges uniformly
there to $\zeta(u^+)$.  This proves the first assertion.  For the inverse
observe that
\[
   \zeta(u^-)\ne s\zeta(v^+)
   \quad\Longleftrightarrow\quad
   s^{-1}\zeta(u^-)\ne\zeta(v^+).
\]
Applying the same argument to
$x_N^{-1}=v^{-N}s^{-1}u^{-N}$ proves the second assertion.
\end{proof}

\begin{theorem}[Convergence-group extension]
\label{thm:convergence-extension}
Let $\Gamma$ be a finitely generated group admitting a non-elementary
convergence action on a compact metrizable space.  If $d=d(\Gamma)$, then
\[
   \kk_{\le d+2}(\Gamma)=0.
\]
\end{theorem}

\begin{proof}
If $\Gamma$ does not have property~$(T)$, the conclusion is immediate.
Assume therefore that it does, and fix a minimal generating tuple
$s_1,\ldots,s_d$.  Apply Lemma~\ref{lem:convergence-schottky} with
$E=\{s_1,\ldots,s_d\}$.  Lemma~\ref{lem:convergence-schottky}\textup{(d)}
gives the bridge conditions
\[
   \zeta(t^+)\ne s_i\zeta(t^-),
   \qquad 1\le i\le d.
\]

Fix $M\ge1$, and write
\[
   b_j=t^jbt^{-j},\qquad
   J_M=\langle b_0,\ldots,b_M\rangle.
\]
As in the proof of Lemma~\ref{lem:correction-blocks}, choose positive
integers $r_i,q_i$ so that the $2d$ intervals
\[
   [r_i,r_i+M],\qquad[-q_i,M-q_i]
\]
are pairwise disjoint and avoid $[0,M]$, and so that
\begin{equation}\label{eq:convergence-correction}
   \zeta(b_{r_i}^-)\ne s_i\zeta(b_{-q_i}^+),
   \qquad 1\le i\le d.
\end{equation}
Such a choice is possible because
\[
   \zeta(b_r^-)\longrightarrow\zeta(t^+),
   \qquad
   \zeta(b_{-q}^+)\longrightarrow\zeta(t^-)
   \qquad(r,q\to\infty),
\]
and the bridge inequalities are open conditions.  Put
\[
   u_i=b_{r_i},\qquad v_i=b_{-q_i},
\]
and, for $0\le j\le M$, put
\[
   p_{ij}=t^ju_i^+=b_{j+r_i}^+,
   \qquad
   q_{ij}=t^jv_i^-=b_{j-q_i}^-.
\]
Let $I_M$ be the union of the above index intervals and put
$B_{I_M}=\langle b_k:k\in I_M\rangle$.  Distinct elements of the Schreier
basis $\{b_k:k\in\mathbb Z\}$ have disjoint fixed-point sets in
$\partial B$.  Hence the displayed endpoints are pairwise distinct and
belong to $\partial B_{I_M}$.  Lemma~\ref{lem:boundary-separation} now
shows that the family $((p_{ij},q_{ij}))_{i,j}$ is $J_M$-separated.

For $N\ge1$, define
\[
   x_i(N)=u_i^Ns_iv_i^N,
   \qquad
   g_{ij,N}=t^jx_i(N)t^{-j}.
\]
By \eqref{eq:convergence-correction},
Lemma~\ref{lem:convergence-sandwich}, and equivariance, these sequences
satisfy \eqref{eq:convergence-relative-dynamics}.  Proposition
\ref{prop:convergence-relative-pp} therefore shows that, for all
sufficiently large $N$,
\[
   H_{M,N}=\left\langle
      b_j,g_{ij,N}:0\le j\le M,\ 1\le i\le d
   \right\rangle
\]
is a nonabelian free group.  Choose one such $N=N(M)$.  The group
$H_{M,N}$ has infinite index in $\Gamma$: otherwise it would inherit
property~$(T)$
from $\Gamma$, whereas a nonabelian free group maps onto $\mathbb Z$ and
cannot have property~$(T)$; see \cite[Theorem~1.7.1]{BHV}.

The set
\[
   S_M=\{t,b,x_1(N),\ldots,x_d(N)\}
\]
generates $\Gamma$, since
\[
   s_i=u_i^{-N}x_i(N)v_i^{-N},
   \qquad 1\le i\le d.
\]
It has cardinality at most $d+2$, and
Lemma~\ref{lem:window}, with
$Y=\{b,x_1(N),\ldots,x_d(N)\}$, gives
\[
   \kk(\Gamma,S_M)\le\sqrt{\frac{2}{M+1}}.
\]
Letting $M$ tend to infinity proves the theorem.
\end{proof}

\begin{corollary}[Hyperbolic and relatively hyperbolic groups]
\label{cor:nonlinear-examples}
The conclusion of Theorem~\ref{thm:convergence-extension} holds for:
\begin{enumerate}[label=(\roman*)]
\item every finitely generated non-elementary word-hyperbolic group;
\item every finitely generated non-elementary subgroup of a
      word-hyperbolic group;
\item every finitely generated non-elementary relatively hyperbolic group.
\end{enumerate}
\end{corollary}

\begin{proof}
A word-hyperbolic group acts as a convergence group on its Gromov
boundary.  More generally, a subgroup of a convergence group acts as a
convergence group after restricting the action to its limit set; in
\textup{(ii)}, ``non-elementary'' means that this limit set has more than
two points.  A
relatively hyperbolic group acts as a convergence group on its Bowditch
boundary.  In each case the hypothesis of
Theorem~\ref{thm:convergence-extension} is precisely the stated
non-elementarity assumption; see \cite{Bowditch-hyp,Yaman}.
\end{proof}

\begin{remark}[Comparison with Osin]
Osin proved that the usual uniform Kazhdan constant of every infinite
word-hyperbolic group vanishes \cite{Osin}.  Theorem
\ref{thm:convergence-extension} strengthens this in the non-elementary
case by imposing the fixed bound $d(\Gamma)+2$ on the cardinality of the
generating sets.
\end{remark}

\begin{remark}[The Osin--Sonkin examples]
There is no conflict with the uniform Kazhdan groups constructed by Osin
and Sonkin \cite{OS}.  Their examples are infinite periodic quotients of
hyperbolic Kazhdan groups, rather than hyperbolic groups themselves.  In
particular, they contain no loxodromic elements.  Since every
non-elementary convergence group contains a loxodromic element, those
groups admit no non-elementary convergence action.
\end{remark}

\begin{remark}[Acylindrically hyperbolic groups]
The argument does not immediately extend to every acylindrically
hyperbolic group.  Such a group contains Schottky free subgroups, but the
boundary of the hyperbolic space on which it acts need not be compact.
Compactness is used in the finite-family relative ping-pong argument.  An
extension to that setting would require an additional uniform separation
statement, perhaps formulated using hyperbolically embedded subgroups.
\end{remark}
}

{
\section*{Acknowledgments}
The first author was partially supported by the National Science
Foundation under grant DMS-2505196.  The second author was partially
supported by the European Research Council (ERC) under the European
Union's Horizon 2020 research and innovation programme (grant agreement
No.~882751, TeStability).

The authors acknowledge the use of OpenAI's ChatGPT in developing the
arguments and preparing the manuscript.  The authors take full
responsibility for the mathematical content.
\par}

\end{document}